\documentclass[a4paper,11pt]{article}
\usepackage{latexsym}
\usepackage{amssymb}
\usepackage{amsmath}
\usepackage{amsthm}
\usepackage[pdftex]{color,graphicx}
\usepackage[numbers]{natbib}
\newtheorem{theorem}{Theorem}
\newtheorem{lemma}[theorem]{Lemma}

\newtheorem{defn}[theorem]{Definition}
\newtheorem{prop}[theorem]{Proposition}

\title{Semimartingale decomposition of convex functions of continuous semimartingales by Brownian perturbation}
\author{Nastasiya F. Grinberg \\ {\it Department of Statistics, University of Warwick}\\{\tt N.F.Grinberg@gmail.com}}

\begin{document}
\maketitle
\newcommand{\sub}{\overline{\nabla}f(x)}
\newcommand{\subo}{\overline{\nabla}f}
\newcommand{\gat}{Df(x)[y]}
\newcommand{\gato}[1]{Df(x)[#1]}
\newcommand{\nc}{\nabla f(x)\cdot y}
\newcommand{\nco}{\nabla f(x)\cdot}
\newcommand{\snc}{\overline{\nabla}f(x)\cdot y}
\newcommand{\snco}{\overline{\nabla}f(x)\cdot}
\newcommand{\sncn}{\overline{\nabla}f_n(x)\cdot y}
\newcommand{\sncno}{\overline{\nabla}f_n(x)\cdot}

In this note we prove that the local martingale part of a convex function $f$ of a $d$-dimensional semimartingale $X=M+A$ can be written in terms of an It\^{o} stochastic integral $\int H(X)dM$, where $H(x)$ is some particular measurable choice of subgradient $\sub$ of $f$ at $x$, and $M$ is the martingale part of $X$. This result was first proved by Bouleau in \cite{Bouleau.81}. Here we present a new treatment of the problem. We first prove the result for $\widetilde{X}=X+\epsilon B$, $\epsilon>0$, where $B$ is a standard Brownian motion, and then pass to the limit as $\epsilon\rightarrow 0$, using results in \cite{Barlow.Protter.90} and \cite{Carlen.Protter.92}. The former paper concerns convergence of semimartingale decompositions of semimartingales, while the latter studies a special case of converging convex functions of semimartingales.
\bigskip

\section{Introduction}

Consider a general convex function $f:\mathbb{R}^d\rightarrow \mathbb{R}$, not necessarily everywhere differentiable. Every differentiable point $x\in\mathbb{R}^d$ has a unique tangential hyperplane, while at non-differentiable points there is a whole set of supporting hyperplanes. For a continuous semimartingale $X$ with decomposition $X=M+A$ we prove that the (local) martingale part of $f(X)$ can be expressed in terms of a stochastic integral of a measurable selection of a subgradient $\subo(X)$ against $M$. For piecewise linear 1-dimensional convex functions this follows from the Meyer-Tanaka formula. For example, for $f(x)=\vert x\vert$ we have $\sub=\hbox{sgn}(x)$, where $\hbox{sgn}(x)=-1$ if $x\leq 0$ and 1 otherwise. So at the origin, which is the only point where derivative is not defined, we can take the supporting line to be $y=-x$. Moreover, since Brownian motion spends zero time in Lebesgue-null sets, we can in fact choose $\overline{\nabla}f(0)$ to be any number in the interval $[-1,+1]$ (corresponding to the possible slopes of supporting lines at 0).\smallskip

The main result of this note is the following\par\smallskip


\begin{theorem}\label{C-theorem 1} Let $f:\mathbb{R}^d\rightarrow \mathbb{R}$ be a convex function and let $X$ be a continuous $\mathbb{R}^d$-valued semimartingale with Meyer decomposition $X_t=X_0+M_t+A_t$ which is defined on filtered probability space $(\Omega, \mathcal{F}, \{\mathcal{F}_t\}_{t\geq 0}, \mathbb{P})$. Then $f(X)$ is again a continuous semimartingale; in particular, its local martingale part is given by

\begin{equation*}
\int^t_0 \subo(X_s)dM_s\ ,\quad \text{locally in $\mathcal{H}^1$}\ ,
\end{equation*}

\noindent{where $\sub$ is some choice of subgradient of $f$ at $x$, such that $\overline{\nabla}f(X_t)$ is $\mathcal{F}_t$-measurable for all $t\geq 0$.}
\end{theorem}

\noindent{The first part of the theorem stating that $f(X)$ is a semimartingale was proved by Meyer \cite{Meyer.76} and later by Carlen and Protter \cite{Carlen.Protter.92}. Meyer just proves that $f(X)$ is a semimartingale, while Carlen and Protter express the martingale and the finite variation process parts of the decomposition in terms of certain limits. Neither of the papers however give an explicit semimartingale decomposition of $f(X)$. In \cite{Bouleau.81}, Bouleau took a step further and proved that at each $x\in \hbox{dom}(f)$ there exists a choice $H(x)$ of a subgradient $\sub$ of $f$ such that the martingale part of the decomposition of $f(X)$ can be expressed as an It\^{o} stochastic integral $\int H(X)dM$. In the follow-up paper \cite{Bouleau.84} he proves the conjecture stated in \cite{Bouleau.81} that  in fact {\it any measurable choice} of $H(x)$ can be used. In this note we are proving the first of the two results using an approach completely different to that in \cite{Bouleau.81}.}\par\smallskip

There are many other papers on extending the It\^{o}'s formula by considering different classes of functions $f$ or stochastic processes, or both.
In \cite{russo.vallois.96}, for example, Russo and Vallois derive It\^{o}'s formula for $\mathcal{C}^1(\mathbb{R}^d)$-functions of continuous semimartingales whose time-reversals are also continuous semimartingales. They also extend the formula to the case of $\mathcal{C}^1(\mathbb{R}^d)$-functions with first order derivatives being H\"{o}lder-continuous with any parameter and the process given by a stochastic flow generated by a so-called $C^0(\mathbb{R}^d,\mathbb{R}^d)$-semimartingale. In both cases the quadratic variation process is expressed in terms of the generalised quadratic covariation process $\langle f'(X), X\rangle_t$ introduced by the authors in an earlier paper \cite{russo.vallois.95} (see also a paper by Fuhrman and Tessitore \cite{fuhrman.tessitore.05}, where authors extend the notion of the generalised quadratic covariation further to the infinite-dimensional case and to non-differentiable functions). In \cite{follmer.protter.shiryayev.95}, F\"{o}llmer, Protter and Shiryayev consider the case of an absolutely continuous function $f$ with a locally square integrable derivative and $X$ a 1-dimensional Brownian motion, for which a version of It\^{o}'s formula is derived with the finite variation part expressed again in terms of the quadratic covariation $\langle f'(B), B\rangle_t$. The multidimensional case (where $f$ belongs to the Sobolev space $\mathbb{W}^{1,2}$) is treated in \cite{follmer.protter.00}. In \cite{Kendall.87}, Kendall discusses a semimartingale decomposition of $r(B)$, where $r$ is a distance function of a Brownian motion on a manifold. The problem tackled in \cite{Kendall.87} is similar to ours as $r$ fails to be differentiable on a set of measure zero, called the cut-locus. It is proved in \cite{Kendall.87} that $r(B)$ is a semimartingale and its canonical decomposition is found explicitly in the sequel \cite{Cranston.Kendall.March.93}.\par\smallskip

The layout of the paper is as follows. In Sections 2 and 3 we introduce some notations and preliminary results concerning convex functions, including some important results on differentiability; in particular, in Section 3 we explain that a proper convex function is everywhere differentiable (i.e. has a unique supporting hyperplane) except on a set of measure zero. Hence, by virtue of observing that a Brownian perturbation of our semimartingale $\widetilde{X}_t^{(\epsilon)}=X_t+\epsilon B_t$ has a probability density at every time $t$, we show that for a convex function $f$ the gradient $\nabla f(\widetilde{X}_t^{(\epsilon)})$ is defined for all $t$ almost everywhere. To show that the martingale part of $f(\widetilde{X}^{(\epsilon)})$ is given by $\int \subo(\widetilde{X}^{(\epsilon)})d\widetilde{M}^{(\epsilon)}$, where $\widetilde{M}^{(\epsilon)}=M+\epsilon B$ and $\subo$ is some measurable choice of a subgradient, we approximate $f$ by a sequence of $\mathcal{C}^2$ convex functions $f_n:\mathbb{R}^d\rightarrow \mathbb{R}$, $n\geq 1$; this is done in Section 5. The martingale part of each $f_n(\widetilde{X}_t^{(\epsilon)})$ is known explicitly from It\^{o}'s formula and is equal to $\int {\nabla} f_n(\widetilde{X}^{(\epsilon)})d\widetilde{M}^{(\epsilon)}$. Convergence of the stochastic integral $\int {\nabla} f_n(\widetilde{X}^{(\epsilon)})d\widetilde{M}^{(\epsilon)}$ to $\int \subo(\widetilde{X}^{(\epsilon)})d\widetilde{M}^{(\epsilon)}$ is ensured by the result of Carlen and Protter \cite{Carlen.Protter.92}. We conclude by proving the convergence $\lim_{\epsilon\downarrow 0} \int \subo(\widetilde{X}^{(\epsilon)})d\widetilde{M}^{(\epsilon)}=\int \subo(X)dM$ in Section 6. Section 4 deals with a special case when $f$ is piecewise linear. By proving a generalised version of Meyer-Tanaka formula we find the local martingale part of $f(X)$ and thus prove Theorem 1 for such $f$. We conclude by giving a particular example of a subgradient that satisfies Theorem 1. \par\bigskip


\section{Convex functions: some notations and results}

In order to prove the main result of this note, we require some notations and results from convex analysis. Proofs of the results stated in this section and more details on convex functions are given in \cite{rockafellar}. See also \cite{giles}.

Let $f$ be any function living on $\mathbb{R}^d$ and taking values in $[-\infty, +\infty]$. At any point $x\in\mathbb{R}^d$ we define the {\it one-directional derivative of $f$ with respect to a vector $y\in\mathbb{R}^d$}, if it exists, as follows
\begin{equation*}
\gat:=\lim_{\lambda\downarrow 0}\frac{f(x+\lambda y)-f(x)}{\lambda}\ .
\end{equation*}

\noindent{The two sided derivative at $x$ in direction $y$ exists if and only if $-Df(x)[-y]$, defined by}
\begin{equation*}
-Df(x)[-y]:=\lim_{\lambda\uparrow 0}\frac{f(x+\lambda y)-f(x)}{\lambda} \ ,
\end{equation*}

\noindent{is also well-defined and}
\begin{equation}\label{c1}
 \gat=-\gato{-y} \ .
 \end{equation}

{Now, if the function $f$ is convex, then the one-directional derivative always exists and, moreover, we may write}
\begin{equation}\label{c2}
\gat=\inf_{\lambda> 0}\frac{f(x+\lambda y)-f(x)}{\lambda}\ .
\end{equation}

\noindent{Furthermore, $\gat$ is positively homogeneous (i.e. $\gato{\lambda y}=\lambda \gat$ for $\lambda\in(0,\infty)$), convex in $y$ with $\gato{0}=0$ \cite[Thm. 23.1]{rockafellar} and}
\begin{equation}\label{c3}
\gat\geq -\gato{-y} \ .
\end{equation}

If for a convex function $f$ defined on $\mathbb{R}^d$ and finite at some $x\in\mathbb{R}^d$ all directional derivatives at $x$ exist, are two-sided and finite then we have (\cite[Thm. 25.2]{rockafellar})
\begin{equation*}
\gat=\langle \nabla f(x), y\rangle,\quad \forall y\in\mathbb{R}^d\ ,
\end{equation*}

\noindent{where}
\begin{equation*}
\nabla f(x):=\left( \frac {\partial f}{\partial x_1}(x),...,\frac{\partial f}{\partial x_d}(x)\right)
\end{equation*}

\noindent{is the \emph{gradient} of $f$ at $x=(x_1,...,x_d)$. Note that $\frac{\partial f}{\partial x_i}(x)=Df(x)[e_i]$, where $e_i$ is the $i^{\text{th}}$ canonical basis vector of $\mathbb{R}^d$.

Of course a general convex function $f$ is not necessarily everywhere differentiable, a simple example being $f(x)=\vert x\vert$ which is not differentiable at $x=0$. We can, however, define a set of {\it subgradients} at each $\lq\lq$troublesome'' point like this. \par\smallskip


\begin{defn}\label{C-defn 2} Let $f:\mathbb{R}^d\rightarrow \mathbb{R}$ be a convex function. A subgradient $\sub$ of $f$ at $x\in\mathbb{R}^d$ is a gradient of an affine hyperplane $h(x)=\alpha+\beta^Tx$, for $\alpha, \beta\in\mathbb{R}^d$, passing through the point $(x, f(x))$ and satisfying
\begin{equation*}h(x')\leq f(x')\end{equation*}

\noindent{for all $x'\ne x$.}
\end{defn}

We say that $h$ is a {\it supporting hyperplane} of $f$ at a point $(x, f(x))$. Clearly, at differentiable points $h$ is unique and is just the tangent of $f$. Conversely, at points where $f$ is not differentiable we can construct infinitely many tangential hyperplanes $h$. The set of all subgradients at $x$ is called the {\it subdifferential} of $f$ at $x$, denoted $\partial f(x)$. A convex function with finite values is {\it subdifferentiable} everywhere. In subsequent sections we will need the following result \par\smallskip


\begin{theorem}\label{C-theorem 3} (\cite[Thm. 23.2]{rockafellar})  Let $f$ be a convex function and $x$ a point at which $f$ is finite. Then $\sub$ is a subgradient of $f$ at $x$ if and only if
\begin{equation}\label{c4}
\gat\geq \langle \overline{\nabla}f(x), y\rangle \qquad \forall y\in\mathbb{R}^d\backslash \{0\} \ .
\end{equation}
\end{theorem}

The theorem above says that a subgradient at $x$ in the direction of $y$ will always be less or equal to the one-sided directional derivative at $x$ with respect to $y$. Relation \eqref{c4} is called the {\it subgradient inequality} and can be used as an alternative definition of a subgradient.
\par

Finally we mention the Lipschitz continuity property of convex functions (see, for example, \cite[Ch. 3.1, Thm. 10]{giles}): if $f$ is a \emph{continuous} convex function on $\mathbb{R}^d$ and $U$ is an open convex subset of $\mathbb{R}^d$, then for all $u\in U$ there exist constants $K>0$ and $\epsilon>0$ such that
\begin{equation*}
\vert f(x)-f(y)\vert \leq K \| x-y\|, \quad \forall x,y\in B_u(\epsilon)\ ,
\end{equation*}

\noindent{where $B_u(\epsilon)$ is an open ball of radius $\epsilon$ centered at $u$ and $\|\cdot\|$ is the usual Euclidean norm.}

\section{ Differential theory of convex functions}

In this section we study differntiability of convex functions and also state and prove certain results concerning convergence of gradients and subgradients of convex functions. In what follows we assume that $f$ is \emph{proper}, i.e. $f(x)<+\infty$ for at least one $x$ and $f(x)>-\infty$ for all $x$. By $\text{dom}f$ we denote the \emph{effective domain} of $f$, that is $\text{dom}f=\{x\in\mathbb{R}^d: f(x)<\infty\}$. We denote by $\text{int}(\text{dom}f)$ the interior of $\text{dom}f$.

Suppose a convex function $f:\mathbb{R}^d\rightarrow\mathbb{R}$ is finite at some point $x\in\mathbb{R}^d$. Then $f$ is differentiable at $x$ if and only if the directional derivative $Df(x)[\cdot]$ is linear on $\mathbb{R}^d$. Moreover, in order for this condition to be satisfied, it suffices that the partial derivatives with respect to the basis vectors of $\mathbb{R}^d$ exist at $x$ (\cite[Thm. 25.2]{rockafellar}). Let us denote by $\mathcal{D}$ the set of points in the domain of $f$ at which the supporting hyperplane is unique, i.e. at which $f$ is differentiable. It is known (\cite[Thm. 25.4]{rockafellar}) that for a proper convex function $f$ the set $\mathcal{D}$ is dense in $\text{int}(\text{dom}f)$ and that its complement in $\text{int}(\text{dom}f)$ \emph{is a set of measure zero}. Consequently any process whose law has a probability density at each time $t>0$ spends time of measure zero in $\mathcal{D}^c$, an important fact we will use in the sequel.

To prove Theorem \ref{C-theorem 1} for a general (continuous and proper but not necessarily differentiable) convex $f$ we will approximate it by a sequence of twice continuously differentiable convex functions $f_n:\mathbb{R}^d\rightarrow\mathbb{R}$, $n\geq 1$, to which we know It\^{o}'s formula can be applied. On top of this, working with convex functions gives us an advantage of being able to deduce from the pointwise convergence of the functions something about the convergence of their corresponding gradients.

\begin{theorem}\label{C-theorem 4} (variation of \cite[Thm. 25.7]{rockafellar})
Let $f$ be a convex function defined on $\mathbb{R}^d$ and $\{f_n\}_{n\geq 1}$ a sequence of smooth convex functions on $\mathbb{R}^d$ such that $\lim_{n\rightarrow\infty}f_n(x)=f(x)$ $\forall x\in\mathbb{R}^d$. Let $\mathcal{D}\subseteq\text{int}(\text{dom}f)$ be the set of points where $f$ is differentiable. Then
\begin{equation}\label{c5}
\lim_{n\rightarrow\infty}\nabla f_n(x)=\nabla f(x)\qquad \forall x\in \mathcal{D} \ .
\end{equation}
\end{theorem}

\begin{proof}
See proof of \cite[Thm. 25.7]{rockafellar}.
\end{proof}

\noindent{This result will be used several times in Sections 5 and 6.}

We next state and prove a result concerning convergence of subgradients of convex functions. Let $f$ be convex; consider a sequence $\{x_n\}_{n\geq 1}$ with $x_n\in \text{int}(\text{dom}f)$, $n\geq 1$, and $x\in \text{int}(\text{dom}f)$ such that $\lim_{n\rightarrow \infty} x_n=x$. Of course in general $\lim_{n\rightarrow \infty}\overline{\nabla} f(x_n)$ need not exist. However, the situation when $x_n=x+\epsilon_n y$ for some $y\in\mathbb{R}^d$ and $\epsilon_n\rightarrow 0$ as $n\rightarrow\infty$, i.e. when $x_n$ approaches $x$ from a single direction $y$, is special. In this case it is known that $\overline{\nabla}f(x_n)$ converges to the part of the boundary of $\partial f(x)$ consisting of points at which $y$ is normal to $\partial f(x)$ \cite[Thm. 24.6]{rockafellar}. Moreover,
\begin{theorem}\label{C-lemma 9}
 Let $f:\mathbb{R}^d\rightarrow\mathbb{R}$ be a convex function. For any $x\in\mathbb{R}^d$, for almost all $y\in S^{d-1}$, where $S^{d-1}$ is the unit sphere in $\mathbb{R}^d$,
\begin{equation*}
\lim_{\epsilon\downarrow 0} \overline{\nabla}f(x+\epsilon y)
\end{equation*}

\noindent{exists, belongs to $\partial f(x)$ and is unique for any selection $\overline{\nabla}f(x+\epsilon y)\in \partial f(x+\epsilon y)$ we may make from the subdifferential of $f$ at $x+\epsilon y$ for any $\epsilon>0$.}
\end{theorem}

\begin{proof} First of all recall that $Df(x)[y]=\lim_{\epsilon\downarrow 0}(f(x+\epsilon y)-f(x))/\epsilon$ is a positively homogeneous function, convex in $y$ with $Df(x)[0]=0$. Let $g(y):= Df(x)[y]$. Hence $\nabla g(\lambda y)$ exists and is unique for all $\lambda>0$ for almost all $y\in\mathbb{R}^d$. Fix $x,y\in \mathbb{R}^d$ and without loss of generality, by adding a suitable affine function to $f$, assume that
\begin{equation*}
f(x)=g(y)=\nabla g(y)=0 \ .
\end{equation*}

We argue by contradiction. If theorem fails then we can find a subsequence $\epsilon_n\rightarrow 0$ and a selection $\overline{\nabla}f(x+\epsilon_ny)\in\partial f(x+\epsilon_ny)$ such that
\begin{equation}\label{c15}
\lim_{n\rightarrow\infty}\overline{\nabla}f(x+\epsilon_ny)=h\ne 0 \ ,
\end{equation}

\noindent{and also a vector $u\in\mathbb{R}^d$ with $\langle h,u\rangle>0$. For such $u$ consider}
\begin{equation*}
\frac{f(x+\epsilon_ny+\epsilon_n\lambda u)-f(x+\epsilon_n y)}{\epsilon_n}
=\lambda\frac{f(x+\epsilon_ny+\epsilon_n\lambda u)-f(x+\epsilon_n y)}{\epsilon_n\lambda} \ .
\end{equation*}

\noindent{Using \eqref{c2} and homogeneity of $g(y)$ the above is greater or equal to}
\begin{equation*}
\frac{\lambda}{\epsilon_n}Df(x+\epsilon_ny)[\epsilon_nu]=\lambda Df(x+\epsilon_ny)[u]
\geq \lambda\langle \overline{\nabla}f(x+\epsilon_n y), u\rangle
=\lambda\langle h, u\rangle + o(1)
\end{equation*}

\noindent{where the last two inequality signs come from expressions \eqref{c4} and \eqref{c15} respectively, and  $o(1)\rightarrow 0$ as $n\rightarrow\infty$. Thus we obtain}
\begin{equation}\label{c16}
\frac{f(x+\epsilon_ny+\epsilon_n\lambda u)-f(x+\epsilon_n y)}{\epsilon_n}\geq \lambda\langle h, u\rangle + o(1) \ ,
\end{equation}

\noindent{On the other hand, since $f(x)=g(y)=0$, we have}
\begin{equation}\label{c17}
\frac{f(x+\epsilon_n y)-f(x)}{\epsilon_n}=\frac{f(x+\epsilon_ny)}{\epsilon_n}=o(1) \ .
\end{equation}

\noindent{Hence combining \eqref{c16} and \eqref{c17} one obtains}
\begin{gather*}\frac{f(x+\epsilon_ny+\epsilon_n\lambda u)-f(x)}{\epsilon_n}
=\frac{f(x+\epsilon_ny+\epsilon_n\lambda u)-f(x+\epsilon_ny)}{\epsilon_n}
+\frac{f(x+\epsilon_ny)-f(x)}{\epsilon_n}\cr
\geq \lambda\langle h,u\rangle +o(1) \ .
\end{gather*}

\noindent{Letting $n\rightarrow \infty$, i.e. $\epsilon_n\rightarrow 0$, the above inequality becomes}
\begin{gather*}
Df(x)[y+\lambda u]=g(y+\lambda u)\geq \lambda\langle h,u\rangle>0\cr
\Rightarrow\quad \frac{g(y+\lambda u)}{\lambda}=\frac{g(y+\lambda u)-g(y)}{\lambda}\geq \langle h,u \rangle>0 \ .
\end{gather*}

\noindent{And so letting $\lambda\rightarrow 0$ one obtains}
\begin{equation*}
\langle \nabla g(y), u \rangle\geq \langle h,u \rangle>0 \ .
\end{equation*}

\noindent{But this contradicts the assumption that $\nabla g(y)=0$.}

\end{proof}

Finally we equip the set of convex functions on $\mathbb{R}^d$ with the topology of uniform convergence on compact sets with the corresponding metric $\rho$, defined by $\rho(f, g) = \sum_{k=1}^\infty 2^{-k}\rho_k(f,g)$ where

\begin{equation*}
\rho_k(f,g)=\frac{ \sup_{\vert x\vert\leq k}\vert f(x) - g(x)\vert}{1+\sup_{\vert x\vert\leq k}\vert f(x) - g(x)\vert} \ .
\end{equation*}

\noindent{In Section 5 we will consider an approximating sequence $\{f_n\}_{n\geq 1}$ of twice continuously differentiable convex functions approximating a general convex function $f$, such that $\lim_{n\rightarrow\infty}\rho(f_n, f)=0$. We will need the following lemma (partly adapted from \cite[Lemma, p. 2]{Carlen.Protter.92})}

\begin{lemma}\label{C-lemma: inequalities} Let $\{f_n\}_{n\geq 1}$ be a sequence of $\mathcal{C}^2$ convex functions on $\mathbb{R}^d$ and let $f$ be a convex function on $\mathbb{R}^d$, such that $\lim_{n\rightarrow\infty}\rho(f_n, f)=0$. Then for any constant $r\geq 0$
\begin{equation}\label{c8}
\sup_n\sup_{\vert x\vert \leq r}\vert \nabla f_n(x) \vert\leq C_r<\infty, \quad \forall r>0\ ,
\end{equation}

\noindent{and}
\begin{equation}\label{c9}
\sup_{\vert x\vert \leq r}\vert\overline {\nabla}f(x)\vert\leq C_r<\infty, \quad \forall r>0 \ ,
\end{equation}

\noindent{where $C_r$ is some constant only depending on $r$, and $\overline{\nabla}f(x)$ is any choice of subgradient $\partial f(x)$.}

\end{lemma}

\begin{proof}
To see why inequality \eqref{c8} is true, first notice that, since $\lim_{n\rightarrow\infty} \rho(f_n,f)=0$, the variation of the convex functions $f_n$ is uniformly bounded in $n$ on $\{\vert x\vert \leq r+1\}$ for any $r>0$. Denote this bound by $C_r$. Let $x_n$ be such that
\begin{equation*}
\nabla f_n(x_n)=\sup_{\vert x\vert \leq r}\vert \nabla f_n(x)\vert
\end{equation*}

\noindent{and let $u_n:=\nabla f_n(x_n)/\vert\nabla f_n(x_n)\vert$. Then}
\begin{multline}\label{c19}
\vert\nabla f_n(x_n)\vert=\langle \nabla f_n(x_n),\frac{\nabla f_n(x_n)}{\vert\nabla f_n(x_n)\vert}\rangle=\langle \nabla f_n(x), u_n\rangle=Df_n(x)[u_n]\\
=\inf_{\lambda>0}\frac{f_n(x_n+\lambda u_n)-f_n(x_n)}{\lambda
} \leq f_n(x_n+u_n)-f_n(x_n) \ .
\end{multline}

\noindent{But, since $\vert x_n+u_n\vert\leq r+1$, the above is less than or equal to $C_r$ for all $n$ and \eqref{c8} follows.}\par\smallskip

Now, since $f_n$ converges to $f$ uniformly on compact sets, we also have $f_n\rightarrow f$ pointwise. Therefore, for any $x, y$ with $\vert x\vert, \vert y\vert<r+1$ the inequality $f_n(x)-f_n(y)\leq C_r$, $\forall n\geq 1$, (which follows since $C_r$ bounds the variation of $f_n$'s) implies $f(x)-f(y)\leq C_r$ by virtue of taking the limit $n\rightarrow\infty$. So, by a calculation similar to \eqref{c19}, we have for any $\overline{\nabla}f(x)\in\partial f(x)$
\begin{equation*}
\vert\overline{\nabla} f(x^*)\vert=\langle \overline{\nabla} f(x^*), u^*\rangle\leq Df(x^*)[u^*]\leq f(x^*+u^*)-f(x^*)\leq C_r \ ,
\end{equation*}
\noindent{where $x^*$ is such that $\overline{\nabla} f(x^*)=\sup_{\vert x\vert \leq r}\vert \overline{\nabla} f(x)\vert$ and $u^*:=\overline{\nabla} f(x^*)/\vert\overline{\nabla} f(x^*)\vert$.}
\end{proof}

\bigskip\bigskip
\section{Piecewise linear convex functions and Meyer-Tanaka formula}

In this section we start our analysis of the martingale part of $f(X)$. However, instead of treating the case of a general convex function $f$, we first prove Theorem 1 in a special case when $f$ is piecewise linear. Using the Meyer-Tanaka formula, we will verify that any piecewise linear convex function of a continuous semimartingale is itself a continuous semimartingale and find the martingale part of the decomposition explicitly.

This result, although not essential, is a nice warm-up before we start dealing with a more general situation in the sections to follow. We refer reader to \cite[Ch. VI.1]{revuz.yor} for a detailed discussion of classical Tanaka and It\^{o}-Tanaka formulas (for $d=1$). One might also find a discussion of convex functions in \cite[Appendix \S 3]{revuz.yor} useful.

We first recall the {\it Meyer-Tanaka formula} (Tanaka formula, if $X=B$ is a standard Brownian motion):


\begin{theorem} (Meyer-Tanaka formula for continuous semimartingales) Let $X$ be a continuous semimartingale. Define the function $\hbox{sgn}(x)$ to be $-1$ if $x\leq 0$ and $1$ otherwise. Then $f(X)$, where $f(x)=\vert x\vert$, is again a semimartingale and, in particular,

\begin{equation*}
\vert X_t\vert =\vert X_0\vert + \int_0^t \hbox{sgn}(X_s)dX_s + L^0_t\ ,
\end{equation*}

\noindent{where $L^0_t$ is the local time of $X$ at $0$.}
\end{theorem}
\par\smallskip

\noindent{Here, by extending the classic Meyer-Tanaka formula, we prove a more general result. Namely, we will prove that any piecewise linear convex function of a continuous semimartingale is itself a continuous semimartingale and find the martingale part of the decomposition explicitly.}\smallskip


\begin{prop} Let $X=(X^1,..., X^d)$ be a continuous semimartingale living on $\mathbb{R}^d$, with $i^{th}$ component having decomposition $X^i_t=X_0^i+M^i_t+A^i_t$, $i\in\{1,...,d\}$. Let $f:\mathbb{R}^d\rightarrow \mathbb{R}$ be a function defined by $f(x)=l_1(x)\vee...\vee l_k(x)$, $x\in\mathbb{R}^d$, where $l_i(x)=\alpha_i+\sum_{j=1}^d\beta_{ij}x_j=\alpha_i+\beta_i^Tx$, for $\alpha_i, \beta_i\in\mathbb{R}^d$, $i\in\{1,...,k\}$, and $x\vee y:=\sup\{x,y\}$. Then $f(X)$ is a semimartingale with decomposition
\begin{equation}\label{c6}
f(X_t)=f(X_0)+\sum_{i=1}^k\int_0^t {\bf 1}_{B_i}(X_s)\beta_i^TdX_s + \frac{1}{2}L_t \ ,
\end{equation}

\noindent{where $B_i= \{x: \min\{k: \sup_j\{l_j(x)\}=l_k(x) \}=i \}$ and $L_t$ is an increasing process, constant on the complement of $\{t: l_i(X_t)=l_j(X_t) \hbox{ for any } i\ne j\}$. In particular, the local martingale part of $f(X)$ is given by}
\begin{equation}\label{c7}
\sum_{i=1}^k \int_0^t {\bf 1}_{B_i}(X_s)\beta_i^TdM_s \ .
\end{equation}

\end{prop}

\par\medskip

\begin{proof}
We prove the proposition for the case when $k=2$ and any $d\geq 1$ and the general case follows by induction. Consider $f(x)=l_1(x)\vee l_2(x)$. Denote $l_1(X_t)=Y_t$ and $l_2(X_t)=Z_t$. Since $X_t$ is a continuous semimartingale so are affine functionals, $Y_t$ and $Z_t$, of $X_t$. Let the corresponding decompositions be $Y=M+A$ and $Z=N+S$. Consider $f(x)=l_1(x)\vee l_2(x)=y\vee z$. We can rewrite $y\vee z$ as follows
\begin{equation*}
y\vee z=\frac{1}{2}\left(\vert y - z\vert + y + z\right) \ .
\end{equation*}

\noindent{Hence, using the differential notation for simplicity, we obtain}
\begin{equation*}
d(Y_t\vee Z_t) =\frac{1}{2}d \left(\vert Y_t-Z_t\vert + Y_t + Z_t\right)=\frac{1}{2}\left(d( \vert W_t\vert) + dY_t +
d Z_t\right) \ ,
\end{equation*}

\noindent{where $W:=Y-Z$, and so $W=(M-N)+(A-S)$. Using Meyer-Tanaka formula the above becomes}
\begin{equation*}
\frac{1}{2}\left(\hbox{sgn}(W_t)dW_t + dL_t^0 + dY_t + d Z_t\right) \ ,
\end{equation*}

\noindent{where $L_t^0$ is the local time of $W$ at $0$. Next}
\begin{multline*}\frac{1}{2}\left(\hbox{sgn}(W_t)d(M_t-N_t) + \hbox{sgn}(W_t)d(A_t-S_t)+d(M_t+A_t)+d(N_t+S_t)+dL^0_t\right)=\\
 =\frac{1}{2}\big[\big(\hbox{sgn}(W_t)+1\big)dM_t-\big(\hbox{sgn}(W_t)-1\big)dN_t+\\
+\left(\hbox{sgn}(W_t)+1\right)dA_t-\left(\hbox{sgn}(W_t)-1\right)dS_t+dL_t^0\big] \ .
\end{multline*}

\noindent{Now $\hbox{sgn}(W_t)=\hbox{sgn}(Y_t-Z_t)={\bf 1}_{[Y_t>Z_t]}-{\bf 1}_{[Y_t\leq Z_t]}$ and so $\hbox{sgn}(W_t)+1=2{\bf 1}_{[Y_t>Z_t]}$ and
$\hbox{sgn}(W_t)-1=-2{\bf 1}_{[Y_t\leq Z_t]}$. Hence we obtain}
\begin{multline*}
d\left(Y_t\vee Z_t\right)={\bf 1}_{[Y_t> Z_t]}dM_t+{\bf 1}_{[Y_t\leq Z_t]}dN_t+{\bf 1}_{[Y_t> Z_t]}dA_t+{\bf 1}_{[Y_t\leq Z_t]}dS_t + \frac{1}{2}dL_t=\\
={\bf 1}_{[Y_t> Z_t]}dY_t+{\bf 1}_{[Y_t\leq Z_t]}dZ_t + \frac{1}{2}dL_t\\
\end{multline*}

\noindent{or}
\begin{equation*}
Y_t\vee Z_t=Y_0\vee Z_0 + \int_0^t {\bf 1}_{[Y_s> Z_s]}dY_s+\int_0^t{\bf 1}_{[Y_s\leq Z_s]}dZ_s+\frac{1}{2}dL_t \ ,
\end{equation*}

\noindent{where $L_t$ is a continuous increasing process, constant on the complement of $\{t: l_1(X_t)=l_2(X_t)\}$. The above expression is exactly \eqref{c6} for $n=2$. Noticing that $x\vee y\vee z=(x\vee y)\vee z$, the general case follows by induction.}

\end{proof}
\par\medskip


Clearly the integrand in \eqref{c7} is a measurable selection of the multivalued map $\partial f(x)$ and so Theorem \ref{C-theorem 1} holds in the special case of convex piecewise linear functions. To illustrate this result we consider our simple example again: for $f(x)=\vert x\vert$ we have $d=1$, $k=2$, $l_1(x)=-x$ and $l_2(x)=x$ and so $B_1=\{x: x<0\}$, $B_2=\{x: x\geq 0\}$ and $L_t$ is an increasing process constant on the complement of $\{t: X_t=0\}$.\smallskip


\section{ Semimartingale decomposition of $f(\widetilde{X}_t)$}

We are now ready to start the analysis of the general case of a convex function $f$ defined over the whole of the Euclidean space $\mathbb{R}^d$. Let $X$ be a continuous semimartingale in $\mathbb{R}^d$ with decomposition $X=M+A$ and defined on some filtered probability space $(\Omega, \mathcal{F}, \{\mathcal{F}_t\}_{t\geq 0}, \mathbb{P})$. Let $(\widetilde{\Omega}, \mathcal{\widetilde{F}}, \{\widetilde{\mathcal{F}}_t\}_{t\geq 0}, \widetilde{\mathbb{P}})$ be some enlargement of this space and let $B$ be an $(\widetilde{\mathcal{F}}_t)$-standard Brownian motion independent of $X$. Define the {\it perturbed process} $\widetilde{X}$ on $(\widetilde{\Omega}, \mathcal{\widetilde{F}}, \{\widetilde{\mathcal{F}}_t\}_{t\geq 0},\widetilde{\mathbb{P}})$ by
\begin{equation*}
\widetilde{X}_t^{(\epsilon)}:=\widetilde{X}_t:=X_t+\epsilon B_t, \qquad \epsilon>0, \quad t\geq 0 \ .
\end{equation*}

\noindent{For simplicity of notation we shall suppress the superscript $(\epsilon)$ wherever possible. For simplicity also but without loss of generality we can assume that $X_0=\widetilde{X}_0=0$.}

In this section we find the martingale part of $f(\widetilde{X}^{(\epsilon)})$ explicitly in order to take the limit as $\epsilon\rightarrow 0$ in the next section and hence prove Theorem \ref{C-theorem 1}.
The reasoning behind adding a small amount of Brownian motion to $X$ is as follows: we know very little about the behaviour of $X$ as it is a general semimartingale. For instance, it can at some times be trivial, i.e. constant. Hence, it might spend positive amount of time in the points where $f$ is not differentiable, that is, where it has more than one supporting hyperplane. To avoid this happening we perturb $X$ by adding $\epsilon B$. Then\smallskip


\begin{lemma}\label{lemma - density of the perturbed process}
$\widetilde{X}_t$ has a probability density at each $t>0$ and, in particular, spends zero time in any null set.
\end{lemma}

\begin{proof} It suffices to prove that $\widetilde{\mathbb{P}}(\widetilde{X}_t\in N)=0$ for any $t>0$ and $N\subset\mathbb{R}^d$ with $Leb(N)=0$. Then it will follow that for all $t> 0$ the law of $\widetilde{X}_t$ under $\widetilde{\mathbb{P}}$ is absolutely continuous with respect to the Lebesgue measure. For any Lebesgue-null set $N$ we have
\begin{equation*}
\widetilde{\mathbb{P}}(\widetilde{X}_t\in N)=\mathbb{E}\left[\widetilde{\mathbb{P}}(\widetilde{X}_t\in N\vert \mathcal{F}_t)\right] \ ,
\end{equation*}

\noindent{where $\mathcal{F}_t=\sigma(\{X_s; 0\leq s\leq t\})$, and we use the tower property of conditional expectation. Next we express $\widetilde{X}_t$ in terms of $X_t$ and $B_t$ and use the fact that $B_t$ is independent of $X_t$, and hence of $\mathcal{F}_t$, to obtain}
\begin{equation*}
\mathbb{E}\big[\widetilde{\mathbb{P}}(X_t+\epsilon B_t \in N\vert \mathcal{F}_t)\big]=\int\widetilde{\mathbb{P}}(x+\epsilon B_t \in N) d\mu_t(x) \ ,
\end{equation*}

\noindent{where $\mu_t$ is the law of $X_t$ (under $\mathbb{P}$). Observe that $\hat{B}_t:=x+\epsilon B_t$ is a Brownian motion started at $x$ with $\langle \hat{B}_t, \hat{B}_t\rangle =\epsilon^2 t$. But we know that Brownian motion hits null-sets with probability zero. Hence, the above integral is equal to zero and the lemma is proved.}\smallskip

\end{proof}

In Section 3 we have seen that $\mathcal{D}^c$, the set of points at which $f$ is not differentiable, is Lebesgue-null. Consequently, by the above lemma, $\widetilde{X}$ spends zero time at those $\lq\lq$ambiguous'' points. Hence, $\nabla f(\widetilde{X})$ is almost surely everywhere defined. Moreover, a particular measurable choice of $\overline{\nabla} f(x)\in\partial f(x)$ at each $x\in\mathcal{D}^c$ is unimportant as it does not change the value of the stochastic integral $\int_0^t \overline{\nabla}f(\widetilde{X}_s)d\widetilde{M}_s$, which we will show is the martingale part of $f(\widetilde{X})$. To do that we approximate $f$ by a sequence of convex \emph{twice continuously differentiable} functions.\par\smallskip

Let $\{f_n\}_{n\geq 1}$ be a sequence of such twice continuously differentiable convex functions on $\mathbb{R}^d$ converging to $f$ with respect to the metric $\rho$ described at the end of Section 3, i.e. $\lim_{n\rightarrow\infty}\rho(f_n,f)=0$. We need to prove that the stochastic integral $\int_0^t \nabla f_n(\widetilde{X}_s)d\widetilde{M}_s$, the martingale part of $f_n(\widetilde{X})$, converges in some sense to $\int_0^t \overline{\nabla} f(\widetilde{X}_s)d\widetilde{M}_s$ for some measurable choice of $\overline{\nabla}f(x)\in\partial f(x)$, and that it is indeed the martingale part of $f(\widetilde{X})$. It turns out that the convergence is in the $\mathcal{H}^1$ norm: for a continuous semimartingale $X$ with decomposition $X=M+A$ we define

\begin{equation*}
\parallel X\parallel_{\mathcal{H}^p}=\parallel \langle M, M\rangle_\infty^{1/2}+\int_0^\infty\vert dA_s\vert \parallel_{L^p} \ .
\end{equation*}

\noindent{The $\mathcal{H}^p$-space consists of all semimartingales $X$ such that $\parallel X\parallel_{\mathcal{H}^p}<\infty$. Once the convergence is established, the fact that $\int\overline{\nabla}f(\widetilde{X})d\widetilde{X}$ is a local martingale part of $f(\widetilde{X})$ will follow from \cite[Thm. 1]{Carlen.Protter.92} of Carlen and Protter.\smallskip

Suppose $\{X^n\}_{n\geq 1}$ is a sequence of continuous semimartingales with the decomposition
$X^n=X^n_0 + M^n + A^n$, such that $\lim_{n\rightarrow\infty}\mathbb{E}[(X^n-X)^*]=0$. Here $X^*=\sup_t\vert X_t\vert$.
Barlow and Protter prove (\cite[Thm. 1]{Barlow.Protter.90}) that under some regularity conditions imposed
on $M^n$ and $A^n$ not only that the limiting process $X$ is again a continuous semimartingale but that there is also convergence of the corresponding
martingale and finite variation process parts of the decompositions.

In \cite[Thm. 1]{Carlen.Protter.92} Carlen and Protter prove that the assumptions of \cite[Thm. 1]{Barlow.Protter.90} are satisfied in the case when the sequence of $\mathcal{C}^2$ convex functions $\{f_n\}_{n\geq 1}$ of a (not necessarily continuous) semimartingale ${X}=M+A$ converges to a convex $f$, thus making the result applicable in our situation.

We are now ready to prove the following


\begin{lemma}\label{C-lemma 8}
The local martingale part of $f(\widetilde{X}_t)$ is given by the limit
\begin{equation}\label{c18}
\lim_{n\rightarrow\infty}\int_0^t \nabla f_n(\widetilde{X}_s)d\widetilde{M}_s=\int_0^t \overline{\nabla} f(\widetilde{X}_s)d\widetilde{M}_s
\end{equation}

\noindent{locally in $\mathcal{H}^1$, where $\overline{\nabla}f(x)\in\partial f(x)$ is some measurable choice of a subgradient of $f$ at $x$.}
\end{lemma}
\smallskip

\begin{proof}
Since for each $n\geq 1$ $f_n$ is a $\mathcal{C}^2$ function, the martingale part of $f_n(\widetilde{X})$ is given by $\int\nabla f_n(\widetilde{X})d\widetilde{M}$, where $\widetilde{M}=M+\epsilon B$.
 The result of Carlen and Protter, applied to our sequence $\{f_n\}_{n\geq 1}$ and the semimartingale $\widetilde{X}$, then ensures that the martingale part of the limiting process $f(\widetilde{X}_{t})$ is given by the limit of $\int \nabla f_n(\widetilde{X})d\widetilde{M}$ as $n$ tends to infinity, locally in $\mathcal{H}^1$. Our aim is to prove that this limit is indeed equal to $\int \overline{\nabla}f(\widetilde{X})d\widetilde{M}$ for some measurable choice of a subgradient $\overline{\nabla}f\in\partial f$.

 We first need to suitably localise our process. Let $B(r)$ be an open ball of radius $r$ and $B(r')$ an open ball of radius $r'$ with $r'>r>0$, both centred at the origin. For all $r,r'>0$ define stopping times $T_r:=\inf \{t:  X_t\notin B(r)\}$ and $\widetilde{T}_{r'}:=\inf\{t: \widetilde{X}_t\notin B(r')\}$ and take $\widetilde{T}=T_r\wedge \widetilde{T}_{r'}$. Assume also that $\widetilde{X}_{t\wedge \widetilde{T}},X_{t\wedge \widetilde{T}}\in\mathcal{H}^1$ for all $t\geq 0$; we know that continuous semimartingales are at least locally in $\mathcal{H}^1$. We consider the stopped process $\widetilde{X}_{t\wedge \widetilde{T}}$. Note that $X_{t\wedge \widetilde{T}}\in B(r)\subset B(r')$ and $\widetilde{X}_{t\wedge \widetilde{T}}\in B(r')$ for all $t\geq 0$. By Lemma \ref{lemma - density of the perturbed process} the law of the localised process $\widetilde{X}_{t\wedge \widetilde{T}}$ under $\widetilde{\mathbb{P}}$ has the density for all $t<\widetilde{T}$; whether $\widetilde{X}_{\widetilde{T}}$ is in $\mathcal{D}$ or not is not important, since it doesn't affect the value of the integrals $\int_0^{\widetilde{T}}\overline{\nabla}f_n(\widetilde{X}_s)d\widetilde{M}_s$, for $n\geq 1$, and $\int_0^{\widetilde{T}}{\nabla}f(\widetilde{X}_s)d\widetilde{M}_s$.

Note that for proving Lemma \ref{C-lemma 8} it would have sufficed to stop $\widetilde{X}$ at $\widetilde{T}_{r'}$. However, in order to be consistent with localisation we will be using to prove Theorem \ref{C-theorem 1} and also to prove Lemma \ref{c: lemma: conditions} below, we use $\widetilde{T}=T_r\wedge \widetilde{T}_{r'}$ instead.

Notice that convergence of a continuous (local) martingale $M$ in $\mathcal{H}^p$ is equivalent to convergence of $\langle M, M\rangle^{1/2}$ in $\mathcal{L}^p$. So, in this case convergence in $\mathcal{H}^p$ implies convergence in $\mathcal{H}^l$ for $1\leq l<p$. In our case it is easier to prove convergence \eqref{c18} in $\mathcal{H}^2$ and then deduce convergence in $\mathcal{H}^1$.
For any measurable selection $\overline{\nabla}f\in\partial f$ and $t>0$ we have
\begin{gather*}\label{c10}
\lim_{n\rightarrow\infty}\big|\big| \int_0^{t\wedge\widetilde{T}} \left(\nabla f_n(\widetilde{X}_s)-\overline{\nabla}f(\widetilde{X}_s)\right)d\widetilde{M}_s\big|\big|_{\mathcal{H}^2}\cr
=\lim_{n\rightarrow\infty}\mathbb{E}\Big[\int^{t\wedge\widetilde{T}}_0\left(
\nabla f_n(\widetilde{X}_s)-\overline{\nabla}f(\widetilde{X}_s)\right)^2d\langle \widetilde{M}, \widetilde{M}\rangle_s \Big]^{1/2} \ .
\end{gather*}

Using inequalities \eqref{c8} and \eqref{c9} we can bound the expression inside the expectation sign above as follows
\begin{equation*}
\begin{split}
\int^{t\wedge\widetilde{T}}_0\left(
\nabla f_n(\widetilde{X}_s)-\overline{\nabla}f(\widetilde{X}_s)\right)^2d\langle \widetilde{M}, \widetilde{M}\rangle_s\leq 4C^2_{r'}\int_0^{t\wedge\widetilde{T}}d\langle\widetilde{M},\widetilde{M}\rangle_s\\
\leq 4C^2_{r'}\langle\widetilde{M},\widetilde{M}\rangle_{\widetilde{T}}<\infty\ ,
\end{split}
\end{equation*}

\noindent{where the quadratic variation $\langle \widetilde{M}, \widetilde{M}\rangle_{t\wedge \widetilde{T}}$ is finite because it is the bracket of a bounded continuous semimartingale $\widetilde{X}_{t\wedge \widetilde{T}}$ (see \cite[Ch. IV, Thm. 1.3]{revuz.yor}). Using dominated convergence theorem we can now take the limit inside the expectation sign and, since the integrand is bounded above by $4C^2_{r'}$, we can also pull the limit inside the integral sign. We can then use almost sure convergence of $\nabla f_n(\widetilde{X}_t)$ to $\nabla f(\widetilde{X}_t)$ for all $\widetilde{X}_t\in\mathcal{D}$ and the fact that particular choices $\overline{\nabla}f(\widetilde{X}_t)\in\partial f(\widetilde{X}_t)$ for $\widetilde{X}_t\in\mathcal{D}^c$ are not charged by the integral to conclude that the limit in question is equal to}
\begin{equation*}
\mathbb{E}\Big[\int^{t\wedge\widetilde{T}}_0
\lim_{n\rightarrow\infty}\left(\nabla f_n(\widetilde{X}_s)-\overline{\nabla}f(\widetilde{X}_s)\right)^2d\langle \widetilde{M}, \widetilde{M}\rangle_s\ \Big]^{1/2}=0 \ .
 \end{equation*}

It follows that $\int^{t\wedge\widetilde{T}}_0 \nabla f_n(\widetilde{X}_s)d\widetilde{M}_s$ converges to $\int^{t\wedge\widetilde{T}}_0 \overline{\nabla} f(\widetilde{X}_s)d\widetilde{M}_s$ in $\mathcal{H}^2$ and, hence, in $\mathcal{H}^1$. This is true for any radii $r'>r>0$ of localisation, and so \eqref{c18} follows.

\end{proof}

We also prove the following lemma concerning the semimartingale decomposition of $f(\widetilde{X})$ which we will require for the proof of Theorem \ref{C-theorem 1}.

\begin{lemma}\label{c: lemma: conditions} Let $\widetilde{N}^{(\epsilon)}$ and $\widetilde{S}^{(\epsilon)}$ be the martingale and the finite variation parts of the semimartingale decomposition of $f(\widetilde{X}^{(\epsilon)})$ respectively. Then for all $\epsilon\leq 1$
\begin{subequations}
\begin{equation}\label{c-(a)}
\mathbb{E}\Big[\sup_{t\leq \widetilde{T}}\vert \widetilde{N}^{(\epsilon)}_t\vert\Big]\leq K_{r,r'}\ ,
\end{equation}
\begin{equation}\label{c-(b)}
\mathbb{E}\Big[\int_0^{\widetilde{T}}\vert d\widetilde{S}^{(\epsilon)}_t\vert\Big]\leq K_{r,r'} \ ,
\end{equation}
\end{subequations}
\noindent{where $K_{r,r'}$ is a constant depending on $r$ and $r'$ and independent of $\epsilon$.}
\end{lemma}

\begin{proof}
The proof largely follows proof of \cite[Thm. 1]{Carlen.Protter.92}: we prove that the sequence of continuous semimartingales $\{f_n(\widetilde{X})\}_{n\geq 1}$ satisfies the conditions of \cite[Thm. 1]{Barlow.Protter.90}, i.e. that
\begin{subequations}
\begin{equation}\label{c-(i)}\
\lim_{n\rightarrow\infty}\mathbb{E}\Big[\sup_{t\leq \widetilde{T}}\vert f_n(\widetilde{X}_t)-f(\widetilde{X}_t)   \vert \Big]=0
\end{equation}
\begin{equation}\label{c-(ii)}
\sup_{n\geq 1}\mathbb{E}\Big[\sup_{t\leq \widetilde{T}}\vert \widetilde{N}^n_t\vert\Big]\leq K_{r,r'}\ ,
\end{equation}
\begin{equation}\label{c-(iii)}
\sup_{n\geq 1}\mathbb{E}\Big[\int_0^{\widetilde{T}}\vert d\widetilde{S}^n_t\vert\Big]\leq K_{r,r'} \ ,
\end{equation}
\end{subequations}

\noindent{where $\widetilde{N}^n$ and $\widetilde{S}^n$, for $n\geq 1$, are the martingale and the finite variation part of the decomposition of $f_n(\widetilde{X})$ respectively. Then \eqref{c-(a)} and \eqref{c-(b)} will follow immediately by \cite[Thm. 1]{Barlow.Protter.90}. The difference from the proof of \cite[Thm. 1]{Carlen.Protter.92} is only in the fact that we need to ensure that for small enough $\epsilon$ the constant $K_{r,r'}$ above can be taken to be independent of $\epsilon$ (this is necessary in order to apply \cite[Thm. 1]{Barlow.Protter.90} to the sequence of semimartingales $\{f(\widetilde{X}^{(\epsilon)})\}_{\epsilon>0}$ in the proof of Theorem 1).}

First of all notice that \eqref{c-(i)} follows from the fact that $\lim_{n\rightarrow\infty}\rho(f_n, f)=0$. Next we consider \eqref{c-(ii)}; for each $n\geq 1$ the martingale part of $f_n(\widetilde{X})$ is given by the stochastic integral $\widetilde{N}^n=\int{\nabla}f_n(\widetilde{X})d\widetilde{M}$. By the Burkholder-Davis-Gundy inequality we have for some constant $p<\infty$
\begin{equation*}
\begin{split}
\mathbb{E}\Big[\sup_{t\leq \widetilde{T}}\vert \widetilde{N}^n_t\vert\Big]&\leq p\mathbb{E}\Big[\langle \widetilde{N}^n, \widetilde{N}^n\rangle_{\widetilde{T}}^{1/2}\Big]\\
&= p\mathbb{E}\Big[\Big(\int_0^{\widetilde{T}}\vert{\nabla}f_n(\widetilde{X}_t)\vert^2 d\langle \widetilde{M}, \widetilde{M}\rangle_t\Big)^{1/2}\Big]\\
&\leq pC_{r'}\mathbb{E}\Big[\langle \widetilde{M}, \widetilde{M}\rangle_{\widetilde{T}}^{1/2}\Big]\ ,
\end{split}
\end{equation*}

\noindent{where the second inequality follows by inequality \eqref{c8} in Lemma \ref{C-lemma: inequalities}. To finish we need to bound $\langle \widetilde{M}, \widetilde{M}\rangle_{\widetilde{T}}$ by some constant independent of $\epsilon$. We have $\langle \widetilde{M},\widetilde{M} \rangle_{\widetilde{T}}=\langle M, M\rangle_{\widetilde{T}}+\epsilon^2 \widetilde{T}$ which for all $\epsilon\leq 1$ is less or equal to $\langle M, M\rangle_{\widetilde{T}}+{\widetilde{T}}$ which is in turn bounded above by $\langle M, M\rangle_{T_r}+T_r$, since $T_r\geq {\widetilde{T}}=T_r\wedge \widetilde{T}_{r'}$. Hence, for all $\epsilon\leq 1$}
\begin{equation*}
\mathbb{E}\Big[\sup_{t\leq \widetilde{T}}\vert \widetilde{N}^n_t\vert\Big]\leq pC_{r'}\mathbb{E}\Big[(\langle {M}, {M}\rangle_{T_r}+ T_r)^{1/2}\Big] \ ,
\end{equation*}

\noindent{where the right-hand side is independent of $\epsilon$ as well as $n$, and so \eqref{c-(ii)} follows.}

The proof of \eqref{c-(iii)} largely mimics the argument in Carlen and Proter \cite[pp. 4-5]{Carlen.Protter.92}, modulo obvious simplifications to allow for the fact that our case is continuous and using Lipschitz continuity of $f$ in $B(r')$.

The assertion of the lemma now follows by \cite[Thm. 1]{Barlow.Protter.90}.
\end{proof}

\par\bigskip


\section{Proof of Theorem 1}

{Finally we need to derive the analogous result for our original object of interest, continuous semimartingale $X$. }\smallskip

\begin{proof}[Proof of Theorem 1] We have $\lim_{\epsilon\downarrow 0}\widetilde{X}^{(\epsilon)}=X$ almost surely and, thus, for a continuous convex $f$, $\lim_{\epsilon\downarrow 0}f(\widetilde{X}^{(\epsilon)})=f(X)$ almost surely. Note that the limit of the process $\widetilde{X}^{(\epsilon)}$ as $\epsilon$ tends to zero lives in the enlarged probability space $(\widetilde{\Omega}, \widetilde{\mathcal{F}}, \{\widetilde{\mathcal{F}}_t\}_{t\geq 0}, \widetilde{\mathbb{P}})$, even though the original process $X$ is defined on $(\Omega, \mathcal{F},\{{\mathcal{F}}_t\}_{t\geq 0}, \mathbb{P})$. We use the same localisation as in the proof of Lemma \ref{C-lemma 8}, i.e. we consider $\widetilde{X}_{t\wedge \widetilde{T}}$ with $\widetilde{T}=T_r\wedge \widetilde{T}_{r'}=\inf \{t:  X_t\notin B(r)\}\wedge \inf\{t: \widetilde{X}_t\notin B(r')\}$, with $r'>r>0$.

Crucially by It\^{o}'s lemma $f(\widetilde{X}^{(\epsilon)})$ is a continuous semimartingale for every $\epsilon>0$. Hence, we can apply the result of Barlow and Protter \cite[Thm. 1]{Barlow.Protter.90} if we can show that the conditions of the theorem are satisfied in our case, i.e. that
\begin{subequations}
\begin{equation*}\label{cc-(i)}\
\lim_{\epsilon\downarrow 0}\mathbb{E}\Big[\sup_{t\leq \widetilde{T}}\vert f(\widetilde{X}^{(\epsilon)}_t)-f({X}_t)   \vert \Big]=0
\end{equation*}
\begin{equation*}\label{cc-(ii)}
\sup_{\epsilon>0}\mathbb{E}\Big[\sup_{t\leq \widetilde{T}}\vert \widetilde{N}^{(\epsilon)}_t\vert\Big]\leq K_{r,r'}\ ,
\end{equation*}
\begin{equation*}\label{cc-(iii)}
\sup_{\epsilon>0}\mathbb{E}\Big[\int_0^{\widetilde{T}}\vert d\widetilde{S}^{(\epsilon)}_t\vert\Big]\leq K_{r,r'} \ ,
\end{equation*}
\end{subequations}

\noindent{where $\widetilde{N}^{(\epsilon)}$ and $\widetilde{S}^{(\epsilon)}$ are the martingale and the finite variation parts of the semimartingale decomposition of $f(\widetilde{X}^{(\epsilon)})$ respectively and $K_{r,r'}$ is some finite constant which only depends on $r$ and $r'$. In view of Lemma \ref{c: lemma: conditions} we need to check only the first of the three conditions above (we can assume that $\epsilon\leq 1$). Using the fact that $f$ is Lipschitz in the ball $B(r')$, we have}
\begin{equation*}
\mathbb{E}\Big[\sup_{t\leq \widetilde{T}}\vert f(\widetilde{X}_t)-f(X_t)\vert\Big]\leq K_{r'}\mathbb{E}\Big[\sup_{t\leq \widetilde{T}}\vert \widetilde{X}_t-X_t\vert\Big] = \epsilon K_{r'}\mathbb{E}\Big[\sup_{t\leq \widetilde{T}}\vert B_t\vert\Big] \ ,
\end{equation*}

\noindent{where $K_{r'}<\infty$ is a Lipschitz constant depending on $r'$. Taking the limit $\epsilon\rightarrow 0$ gives the desired result.}
Together with expressions \eqref{c-(a)} and \eqref{c-(b)} of Lemma \ref{c: lemma: conditions} this ensures that the conditions of \cite[Thm. 1]{Barlow.Protter.90} are satisfied in our case. From Lemma \ref{C-lemma 8} we know that for each $\epsilon>0$ the martingale part of $f(\widetilde{X}^{(\epsilon)})$ is equal to $\widetilde{N}^{(\epsilon)}=\int\overline{\nabla}f(\widetilde{X}^{(\epsilon)})d\widetilde{M}^{(\epsilon)}$; it now follows immediately that the martingale part of $f(X)$ is given by the limit as $\epsilon\rightarrow 0$ of $\widetilde{N}^{(\epsilon)}$, locally in $\mathcal{H}^1$. All is left to prove now is that this limit is given by $\int \overline{\nabla}f(X)dM$ for some measurable choice of $\overline{\nabla}f(x)\in\partial f(x)$, i.e. that for all $t>0$

\begin{equation}\label{c11}
\lim_{\epsilon\downarrow 0}  \int_0^t \overline{\nabla}f(\widetilde{X}^{(\epsilon)}_{s\wedge {\widetilde{T}}})d\widetilde{M}^{(\epsilon)}_{s\wedge {\widetilde{T}}}=\int_0^{t} \overline{\nabla}f({X}_{s\wedge T_r})dM_{s\wedge T_r}
\end{equation}

\noindent{in $\mathcal{H}^1$ for all $r'>r>0$.}

Proving the above convergence will require us to consider the limit of  $\overline{\nabla}f(\widetilde{X}^{(\epsilon)}_{t\wedge {\widetilde{T}}})$ as $\epsilon$ tends to 0.
From Theorem \ref{C-lemma 9} we know that for all $t\geq 0$ for almost all values of $B_t$ the limit $\lim_{\epsilon\downarrow 0}\overline{\nabla}f(X_t+\epsilon B_t)$ exists and belongs to $\partial f(X_t)$. Denote this limit by $\overline{\nabla}f(X_t)$. Also for any path of $X$ and $B$ for small enough $\epsilon$, i.e. eventually for all $\epsilon$, we have $T_r<\widetilde{T}_{r'}$. That is ${\widetilde{T}}=T_r\wedge \widetilde{T}_{r'}\rightarrow T_r$ as $\epsilon\rightarrow 0$ a.s. and so
\begin{equation}\label{c22}
\lim_{\epsilon\downarrow 0}\overline{\nabla}f(X_{t\wedge {\widetilde{T}}}+\epsilon B_{t\wedge {\widetilde{T}}})=\overline{\nabla}f(X_{t\wedge T_r})\quad \hbox{a.s.} \ .
\end{equation}

Again we consider convergence in $\mathcal{H}^2$ first, and convergence in $\mathcal{H}^1$ follows. We have, using the fact that $\lim_{\epsilon\downarrow 0}\widetilde{M}_{t\wedge {\widetilde{T}}}=\lim_{\epsilon \downarrow 0}(M_{t\wedge \widetilde{T}}+\epsilon B_{t\wedge \widetilde{T}})=\lim_{\epsilon\downarrow 0}M_{t\wedge {\widetilde{T}}}$ a.s.
\begin{multline}\label{c12}
\lim_{\epsilon\downarrow 0}\big|\big|  \int_0^t \overline{\nabla}f(\widetilde{X}_{s\wedge {\widetilde{T}}})d\widetilde{M}_{s\wedge {\widetilde{T}}}-\int_0^{t} \overline{\nabla}f({X}_{s\wedge T_r})dM_{s\wedge T_r}\big|\big|_{\mathcal{H}^2}\cr
=\lim_{\epsilon\downarrow 0}\mathbb{E}\big[ \int_0^{\widetilde{T}} \overline{\nabla}f(\widetilde{X}_s)^2d\langle M_{s},M_{s}\rangle +  \int_0^{T_r} \overline{\nabla}f(X_s)^2d\langle M_{s},M_{s}\rangle\big.\cr
\Big.-2\int_0^{\infty}\overline{\nabla}f(\widetilde{X}_{s\wedge {\widetilde{T}}})\overline{\nabla} f(X_{s\wedge T_r})d\langle M_{s\wedge {\widetilde{T}}}, M_{s\wedge T_r}\rangle\big]^{1/2} \ .
\end{multline}

Once again we can use Lemma \ref{C-lemma: inequalities} to see that the first integrand in \eqref{c12} is bounded above by $C_{r'}^2<\infty$, while the third integrand is bounded above by $C_{r'}C_r<\infty$. Thus we have
\begin{equation*}
\int_0^{\widetilde{T}} \overline{\nabla}f(\widetilde{X}_s)^2d\langle M_{s},M_{s}\rangle\leq  C_{r'}^2\langle M_{\widetilde{T}}, M_{\widetilde{T}}\rangle<\infty
\end{equation*}

\noindent{and}
\begin{equation*}
\int_0^{\infty}\overline{\nabla}f(\widetilde{X}_{s\wedge {\widetilde{T}}})\overline{\nabla} f(X_{s\wedge T_r})d\langle M_{s\wedge {\widetilde{T}}}, M_{s\wedge T_r}\rangle\leq C_{r'}C_r\langle M_{T_r},M_{T_r}\rangle<\infty\ ,
\end{equation*}

\noindent{where we use the fact that $\langle M, M\rangle_{t\wedge T_r}$, resp. $\langle M, M\rangle_{t\wedge \widetilde{T}}$, is finite being the bracket of the bounded continuous semimartingale $X_{t\wedge T_r}$, resp. $X_{t\wedge \widetilde{T}}$. Appealing to the dominated and bounded convergence theorems we can interchange the limit in \eqref{c12} with the expectation and the integration signs respectively.
Convergence \eqref{c22} and the fact that ${\widetilde{T}}\rightarrow T_r$ a.s. then lead us to conclude that the limit \eqref{c12} is equal to 0 and so we obtain \eqref{c11}. Noticing that the above is true for all $r'>r>0$ concludes the proof.}

\end{proof}


\paragraph{Example.}
As was mentioned before, in \cite{Bouleau.84}, Bouleau has proved that {\it any} measurable choice of subgradient $\subo(X_t)$ works for the stochastic integral of Theorem \ref{C-theorem 1}. A function

\begin{equation}\label{c13}
\overline{\nabla}^e f(x)=\lim_{\theta\downarrow 0}\mathbb{E}[\overline{\nabla} f(x+\theta N)] \ ,
\end{equation}

\noindent{where $N$ is a standard $d$-dimensional Gaussian random variable, is a particular example. $\overline{\nabla}^e f(x)$ can be regarded as a sort of an average of (sub)gradients within the vicinity of $x$. To verify that it does indeed define a subgradient of $f$ at each $x\in \mathbb{R}^d$ we check the subgradient inequality \eqref{c4} of Theorem \ref{C-theorem 3}. For any $y\in\mathbb{R}^d\backslash\{0\}$ we have}

\begin{equation}\label{c20}
 \langle \overline{\nabla}^e f(x), y \rangle =\langle \lim_{\theta \downarrow 0}\mathbb{E}[\subo(x+\theta N)], y\rangle
 = \lim_{\theta \downarrow 0}\mathbb{E}[\langle\subo(x+\theta N), y\rangle]\ .
\end{equation}


\noindent{Now, by the Lipschitz property of $f$ and by the subgradient inequality \eqref{c4} we have}

\begin{gather*}
\langle \overline{\nabla}f(x+\theta N),y\rangle\leq D(x+\theta N)[y]=\inf_{\lambda>0}\frac{f(x+\theta N +\lambda y)-f(x+\theta N)}{\lambda}\\
\leq f(x+\theta N+y)-f(x+\theta N)\leq K\vert y\vert
\end{gather*}

\noindent{for some Lipschitz constant $K<\infty$ depending on $x$ and $N$. Appealing to the bounded convergence theorem now allows us to take the limit inside the expectation in equation \eqref{c20} above}

\begin{equation}\label{c14}
\langle \overline{\nabla}^e f(x), y \rangle = \mathbb{E}[\langle\lim_{\theta \downarrow 0}\subo(x+\theta N), y\rangle] \ .
\end{equation}

{But by Theorem \ref{C-lemma 9} $\lim_{\theta\downarrow 0}\overline{\nabla}f(x+\theta N)$ exists, is unique and belongs to $\partial f(x)$ for almost all $N$. Denote this limit by $\overline{\nabla}^*f(x)$. Then \eqref{c14} is equal to}

\begin{equation*}
\mathbb{E}[\langle \overline{\nabla}^*f(x), y \rangle]\leq \mathbb{E}[Df(x)[y]]=Df(x)[y] \ .
\end{equation*}

\noindent{Hence, we have $\langle \overline{\nabla}^e f(x), y\rangle \leq Df(x)[y]$ for any $y\in\mathbb{R}^d\backslash\{0\}$ for all $x$, and so $\overline{\nabla}^ef(x)$ is a well-defined subgradient of $f$. }

\medskip

\noindent{\large {\bf Acknowledgements.} This note is a part of my PhD thesis and, therefore, I would like to thank my supervisor Prof Wilfrid Kendall for all the help he gave me and acknowledge funding from the Statistics department of Warwick University which supports my studies. I would also like to thank Michel \'{E}mery for helpful discussions during Probability Summer School at Saint-Flour in July 2008 and Larbi Alili and the anonymous referee for helpful comments on the earlier versions of this note.

\bibliographystyle{plainnat}
\bibliography{ref}

\end{document}